\documentclass[11pt]{amsart}
\usepackage{amsmath, amssymb}
\usepackage{amsfonts}
\usepackage{mathrsfs}
\usepackage[arrow,matrix,curve,cmtip,ps]{xy}
\usepackage{graphicx}
\usepackage{amscd}
\usepackage{amsthm}
\usepackage{amsfonts}
\usepackage{xypic}
\usepackage{latexsym}
\usepackage{bm}
\usepackage{stmaryrd}
\usepackage{geometry}
\usepackage{setspace}
\usepackage{url}
\usepackage{color}
\xyoption{all}

\allowdisplaybreaks
\numberwithin{equation}{section}

\newtheorem{thm}{Theorem}[section]

\newtheorem{cor}[thm]{Corollary}

\theoremstyle{definition}
\newtheorem{exam}[thm]{Example}
\newtheorem{defi}[thm]{Definition}

\theoremstyle{plain}

\newtheorem{lem}[thm]{Lemma}

\newcommand{\abe}{\alpha,\beta,\epsilon}
\newcommand{\sxy}{(S;X,Y)}

\newtheorem*{theorem*}{Theorem}
\theoremstyle{remark}

\begin{document}
\title[3-manifolds as permutations]{Understanding 3-manifolds in the context of permutations}

\author{Karoline Null}

\address{University of Tennessee at Martin}
\email{knull@utm.edu}

\thanks{This paper is from the author's doctoral dissertation.}

\date{\today}

\subjclass[2010]{57N65, 57M05, 57M27, 57M60, 22F30, 22F50, 22F05, 20B10, 22C05}

\keywords{3-manifolds, Heegaard Diagrams, permutations}

\begin{abstract}
We demonstrate how a 3-manifold, a Heegaard diagram, and a group presentation can each be interpreted as a pair of signed permutations in the symmetric group $S_d.$ We demonstrate the power of permutation data in programming and discuss an algorithm we have developed that takes the permutation data as input and determines whether the data represents a closed 3-manifold.  We therefore have an invariant of groups, that is given any group presentation, we can determine if that presentation presents a closed 3-manifold. 
\end{abstract}

\maketitle

\section{Introduction}\label{1}

Since there inception, three-manifolds have been investigated through Heegaard diagrams, splittings, and groups. Translating a 3-manifold to a diagram provides a nice 2-dimensional means of dealing with a difficult 3-dimensional object.  Working with a 3-manifold in terms of groups is also a way of translating the problem from topology into algebra, allowing another field of techniques to be applied to this study.  

We introduce a technique to translate between group presentations, diagrams (and hence Heegaard splittings) and pairs of signed permutations. This was first done by Montesino \cite{Mont} and then by Hempel \cite{hempelpaper} for positive Heegaard diagrams, but little has been done with this approach since then, no doubt because of the tedious calculations that were required. With the now nearly universal availability of fast computers, permutations should be revisited as they provide a quick computational approach to dealing with 3-manifolds.  Permutation data is amenable to algorithmic sorting and checking techniques, thus with the newfound availability of fast computing power, we revisit permutation data as a way of encoding Heegaard diagrams, and hence 3-manifolds.  It has been shown (see \cite{thesis}) that the problem of deciding whether a group presentation $P$ presents a closed 3-manifold is recursively enumerable.  In this paper, we shall determine the decidability of whether a fixed $P$ presents a 3-manifold group.

Every presentation is determined by a family of Heegaard diagrams. We want to decide whether the presentation does indeed present the fundamental group of a 3-manifold determined by one of these Heegaard diagrams.  We explicitly state the characteristics of this family of diagrams associated to a trivially reduced presentation.  These are the classes to examine, finite and infinite.  The infinite class of diagrams is treated in \cite{thesis}.  The finite class, made up of every possible curve re-ordering on the diagram determined by $P$, will be treated here. Any diagram determines a finite number (or a \textit{family}) of signed permutation pairs, which are in one-to-one correspondence with this finite class.  Beginning with only a pair of signed permutations, we determine whether the associated 3-manifold is closed (Lemma~\ref{BdX}).

We have a method for producing a set of signed permutation data to encode a presentation (\S\ref{Ptoabe}), and developed an algorithm that produces this finite family and determines if any diagram from the finite family results in $S-X$ being planar (\S\ref{ribbon}). The algorithm is available upon request.

\section{Preliminaries}\label{2}

\begin{defi}
Let $B^n$ denote the unit ball $\{x\in {\mathbb{R}}^n: ||x||\leq 1\},$ and $S^{n-1}$ denote the unit sphere $\{x\in {\mathbb{R}}^n: ||x||= 1\}.$ We call a space homeomorphic to $B^n$ an \textit{$n$-cell}, and a space homeomorphic to $S^{n-1}$ an \textit{$(n-1)$-sphere}.
\end{defi}

A (topological) \textit{$n$-manifold} is a separable metric space, each of whose points has an open neighborhood homeomorphic to either ${\mathbb{R}}^n$ or ${\mathbb{R}}^{n}_{+}=\{x\in {\mathbb{R}}^n: x_n\geq 0\}.$ The \textit{boundary} of an $n$-manifold $M,$ denoted $\partial M,$ is the set of points of $M$ having neighborhoods homeomorphic to ${\mathbb{R}}^{n}_{+}.$  By invariance of domain, $\partial M$ is either empty or an $n-1$ dimensional manifold and $\partial \partial M=\emptyset$ \cite{LB}. A manifold $M$ is \textit{closed} if $M$ is compact with $\partial M=\emptyset$ and the manifold is \textit{open} if $M$ has no compact component and $\partial M\neq \emptyset.$

A \textit{compression body} $V$ is obtained from a connected surface $S$ by attaching 2-handles to $S\times\{0\}$ and capping off any 2-sphere boundary components with 3-handles. We define $$\partial_+V := S\times\{1\}$$ and $$\partial_-V=\partial V-\partial_+V,$$ the latter of which is also the result of surgery on $S\times \{0\}.$  A \textit{handlebody} is a compression body in which $\partial\_V$ is empty. Throughout this work, we assume all manifolds are oriented. Of interest are 3-manifolds, because every compact, oriented 3-manifold has a splitting (see \cite{hempelbook} for a proof).

\begin{defi}
A \textit{$($Heegaard$)$ splitting} is a representation of a connected 3-manifold $M$ by the union of two compression bodies $V_X$ and $V_Y,$ with a homeomorphism taking $\partial_+ V_X$ to $\partial_+ V_Y.$ The resulting 3-manifold can be written $M=V_X\cup_S V_Y,$ where $S$ is the surface $\partial_+ V_X=\partial_+ V_Y$ in $M.$
\end{defi}

We call $S$ the \textit{splitting surface} and $g(S)$ the \textit{genus of the splitting.} As the Lens Spaces (genus one) are effectively classified \cite{py}, we will only be considering splittings of genus $\geq 2.$  

Before we formally define diagrams, we might do well to point out that one way of viewing diagrams is as a tool for splittings. Suppose we have a splitting of a 3-manifold $M=V_X\cup_S V_Y.$ A diagram shows the attaching curves for the 2-handles of $V_X$ and $V_Y.$ Every compact, oriented, connected 3-manifold has a splitting, and for each splitting many different curve sets could be chosen to determine the compression bodies. Thus, every 3-manifold can be studied through two-dimensional diagrams.

However, we do not need to begin with a splitting and move to the diagram. It is important to consider diagrams abstractly, since throughout this work we will begin with diagrams and determine properties of the associated 3-manifold. 

\begin{defi}
A \textit{diagram} is an ordered triple $\sxy$ where $S$ is a closed, oriented, connected surface and $X:=\{X_1,\ldots,X_m\}$ and $Y:=\{Y_1,\ldots,Y_n\}$ are compact, oriented 1-manifolds in $S$ in relative general position and for which no component of $S-(X\cup Y)$ is a \textit{bigon} --- a disc whose boundary is the union of an arc in $X$ and an arc in $Y.$
\end{defi}

\begin{figure}[ht!]\label{Hdiagram}
   \begin{center}
            \includegraphics[scale=.5]{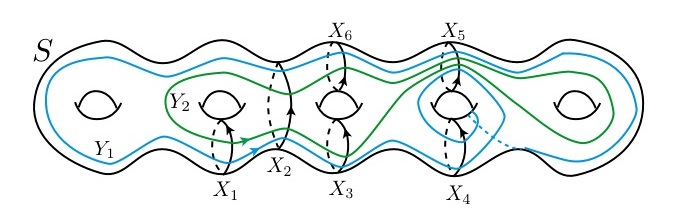}
        \caption{An example of a diagram $D=(S;X,Y)$ with $g(S)=5,$ $m=6,$ and $n=2$.}
    \end{center}
\end{figure}

This definition allows $X$ (or $Y$) to have \textit{superfluous curves}, a subset of components of $X$ (or $Y$) which could bound a planar surface in $S.$ We allow this because there is a correspondence between diagrams and presentations, under which the diagram for a 3-manifold associated to the permutations may have superfluous curves.

Two diagrams $\sxy$ and $(S^*;X^*,Y^*)$ are equivalent provided $g(S)=g(S^*)$ and there is a homeomorphism between surfaces, taking $X$ to $X^*$ and $Y$ to $Y^*.$  An \textit{arc} is a component of $Y - X$ on the surface of $S.$ 

Given a diagram, the manifold $M$ can be recovered from the diagram as follows.  For each $i=1,\ldots, m$ attach a copy of $B^2\times I$ to $S\times [0,1]$ by identifying $\partial B^2\times I$ with a neighborhood of $X_i$ in $S\times\{0\}\subset S\times[0,1].$  For each $i=1,\ldots, n$ attach a copy of $B^2\times I$ to $S\times [0,1]$ by identifying $\partial B^2\times I$ with a neighborhood of $Y_i$ in $S\times\{1\}\subset S\times[0,1].$  The resulting manifold, $M_1,$ has a 2-sphere boundary component for each planar region in $S-X$ and $S-Y.$ Obtain $M$ by attaching a copy of $B^3$ to each 2-sphere boundary component of $M_1.$  We will use this understanding of a diagram throughout this paper, viewing a diagram as giving the splitting surface sitting in a 3-manifold, with $X$ and $Y$ bounding discs on either side of $S.$ 

A diagram also determines a presentation.     
\begin{defi}
Given a diagram $D,$ the \textit{presentation determined by $D$}, denoted $P(D),$ is a finite group presentation with one generator $x_i$ for each component $X_i\in X,$ and one relator for each component of $Y_i\in Y$, defined by recording the intersection with each $X_i$, and performing any trivial reductions. That is, each relator is obtained as $r_i:=x_{i1}^{\epsilon_1}x_{i2}^{\epsilon_2}\ldots x_{ik}^{\epsilon_k},$ where the curve $Y_i$ crosses $X_{i1},X_{i2},\ldots,X_{ik}$ in order with crossing numbers $\epsilon_i$ (see Figure 2).

\begin{figure}[ht!]\label{epsilon}  
   \begin{center}
            \includegraphics[scale=.4]{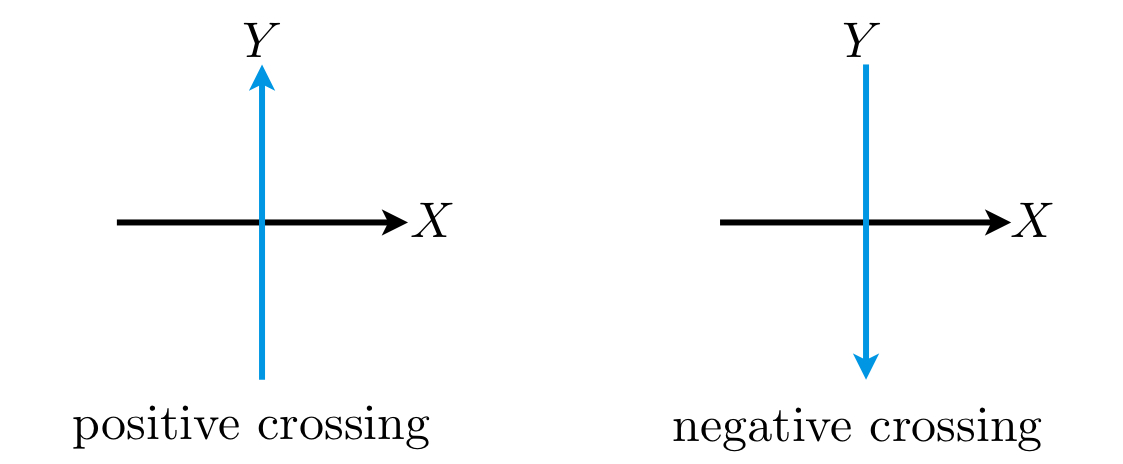}
       \caption{A positive crossing with $\epsilon = 1$ (left) and a negative crossing with $\epsilon = -1$ (right)}
    \end{center}
\end{figure}

When relating this to $\pi_1(M)$ we regard $x_i$ as a curve in $S$ which crosses $X_i$ with a positive crossing number and which crosses no other $X_j$. Instances of $x_ix_i^{-1}$ (or $x_i^{-1}x_i$) that appear in the presentation determined by a diagram will often not be replaced with $1$ in a relator, unless otherwise specified.
\end{defi}

\begin{exam}
Consider the diagram $D=\sxy$ in Figure 1. As $X=\{X_1, X_2, X_3, X_4,$ $X_5, X_6\},$ the presentation $P(D)$ has six generators $\{x_1, x_2, x_3, x_4, x_5, x_6\}.$ As $Y=\{Y_1, Y_2\},$ $P(D)$ has two relators. We record each by flowing along the curve and recording each generator encountered with a superscript of $1$ if the crossing was positive, and $-1$ if the crossing was negative (see Figure 2). Thus 

$P(D) = \langle x_1, x_2, x_3, x_4, x_5, x_6 \, :\, x_1^{-1}x_2^{-1}x_3^{-1}x_4^{-1}x_5\,x_4^{-1}x_5\,x_6\,x_2,\, \, x_1\,x_2^{-1}x_3^{-1}x_5^{-1}x_5\,x_6\,x_2\rangle.$
\end{exam}

To force the construction of $P(D)$ from $D$ to be well-defined, we set the convention that the diagram $\sxy$ is an ordered triple, where $X$ will always correspond to the generating set and $Y$ will always correspond to the set of relators. The presentation determined by a diagram is unique up to inversion and cyclic reordering.  It is well known that the group presented by the presentation, denoted $|P(D)|,$ is isomorphic to the fundamental group of the 3-manifold $M$ determined by the presentation provided $M$ is closed and $X$ is a complete meridian set (see \cite{hempelbook} for a proof).  

\section{How a Heegaard diagram determines a set of oriented permutations}\label{3}			

Let $\sxy$ be a diagram for a 3-manifold. The simple, closed, oriented curves of $X$ intersect with the simple, closed, oriented curves of $Y$ in $d$ points. By numbering the intersection points $1$ through $d,$ we can encode each $X_i$ of $X$ and each $Y_j$ of  $Y$ by listing the numbered intersection points in the order in which they are encountered when flowing along the orientated curve. For each $X_i$ we denote the ordered cycle of intersection numbers as $\alpha_i$ and for each $Y_j$ we denote the ordered cycle as $\beta_j,$ making it possible to interpret $\alpha=\alpha_1\alpha_2\ldots\alpha_m$ and $\beta=\beta_1\beta_2\ldots\beta_n$ as elements of $S_d,$ the symmetric group on $d$ elements. We let $c(\alpha)$ denote the number of cycles in a permutation, so $X$ and $Y$ will have $c(\alpha)$ and $c(\beta)$ components respectively (i.e. one simple closed curve on the surface corresponds to one cycle in the respective permutation). We let $|\alpha_i|$ denote the length of the cycle $\alpha_i.$ Since the curves in $(X\cup Y)$ are oriented, at each of the $d$ points we associate an intersection number $\pm 1$ (see Figure 2) via the intersection function $\epsilon.$

\begin{defi}
The \textit{intersection function} $\epsilon:\{1,2,\ldots,d\}\rightarrow \{1, -1\}$ indicates $X$ and $Y$ have a positively-oriented crossing at $i$ if $\epsilon(i)=1,$ and indicates a negatively-oriented crossing if $\epsilon(i)=-1.$  We call $\epsilon(i)$ the \textit{intersection number at $i.$} For expediency, $\epsilon$ is often written as a $d$-tuple of 1's and $-1$'s.
\end{defi}

We call the triple $(\abe)$ a \textit{permutation data set} for $\sxy.$ Notice that the permutation data set associated to a splitting is unique up to renumbering, corresponding to being unique up to conjugation of $\alpha$ and $\beta$ in $S_d.$

Let $X_i$ have $k_i$ intersection points for $1\leq i\leq m,$ such that $\sum k_i=d.$ Our convention is to label the intersection points of $X_1$ consecutively with $1,2,\ldots,k_1;$ to label the intersection points of $X_2$ consecutively with $(k_1+1), (k_1+2),\ldots, (k_1+k_2);$ and so on, labeling the intersection points of $X_m$ consecutively with $K+1, K+2,\ldots, d$ where $K:= k_1+k_2+\ldots+ k_{m-1}.$

\begin{exam}
Consider the diagram $\sxy$ shown in Figure~\ref{4.2}. Label the intersection points consecutively, beginning on $X_1.$ Then we have 
\begin{align*}
\alpha&=\alpha_1\alpha_2=(1,2)(3,4,5,6),\\
\beta&=\beta_1\beta_2=(1,6,4)(2,3,5),\\
\epsilon&=(1,1,-1,-1,-1,-1).
\end{align*}

\begin{figure}[ht!]  
   \begin{center}
            \includegraphics[scale=1.7]{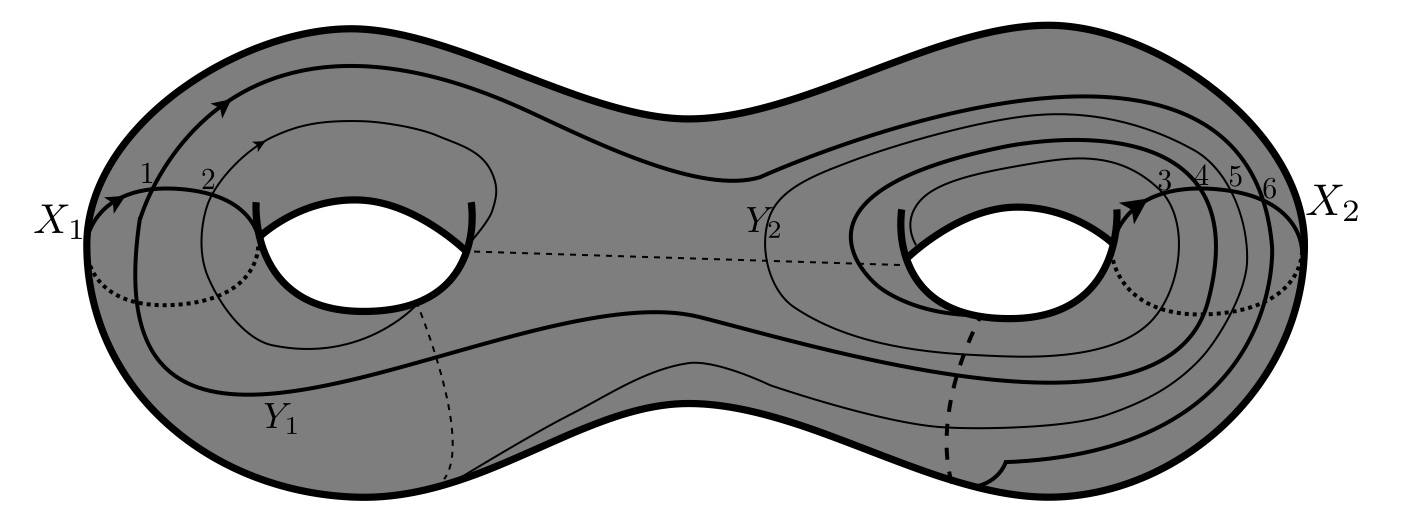}
        \caption{A diagram with intersection points labeled}\label{4.2}
               \end{center}
        \end{figure}

\end{exam}

\section{How a permutation data set determines a presentation}\label{4}

We demonstrate how to construct $P$ directly from $(\abe)$ without needing to create the diagram as an intermediate step. When necessary, we use the notation $P_{\abe}$ to mean the presentation determined from the permutation data set $(\abe).$  

When given a diagram, convention requires the $X$-curves to correspond to the generators and the $Y$-curves to correspond to the relators.  Similarly, we consider the cycles of $\alpha$ to be the generators and the cycles of $\beta$ to encode the relators. 


Let $\alpha,\,\beta$ be in $S_d,$ with $c(\alpha)=m,$ $c(\beta)=n.$ Define the map
$$a:\{1,2,\ldots,d\}\rightarrow\{1,2,\ldots,m\} \textrm{ such that } j \in\alpha_{a(j)}.$$
The group presentation $P_{\abe}$ is written 
$P=\langle \alpha_1,\alpha_2,\ldots,\alpha_m: r_1,r_2,\ldots,r_n\rangle.$ Each $\beta_i:=(i_1,i_2,\ldots,i_k)$ of $\beta$ determines the relator
$$r_i:=\alpha_{a(i_1)}^{\epsilon(i_1)}\alpha_{a(i_2)}^{\epsilon(i_2)}\ldots\alpha_{a(i_k)}^{\epsilon(i_k)}.$$

\begin{exam}
We demonstrate how to determine $P_{\abe}$ from $(\abe).$ Let

\begin{align*}
\alpha&=(1,2,3,4)(5,6,7)(8,9,10,11),\\
\beta &= (1,6)(2,4,11)(8,5,10)(9,3,7),\\ 
\epsilon&=(1, 1, -1, -1, -1, 1, -1, -1, 1, 1, 1).
\end{align*} 

As $c(\alpha)=3,$ and $c(\beta)=4,$ $P_{\abe}$ will have 3 generators and 4 relators:
 $$P_{\abe}=\langle x_1,x_2,x_3:\, x_1x_2,\, x_1x_1^{-1}x_3,\, x_3^{-1}x_2^{-1}x_3,\, x_3x_1^{-1}x_2^{-1}\rangle.$$
  \end{exam}

There is no ambiguity when we  begin with a permutation data set, since $P_{\abe}$ is required to use precisely $(\abe),$ so if $\alpha, \beta\in S_d,$ then $deg_A(P_{\abe})=d$ and no trivial reductions are performed. That is, $P_{\abe}$ need not be reduced as written.

However, when we begin with a presentation as in \S\ref{Ptoabe}, $P$ is always assumed to be trivially reduced unless otherwise stated. Thus, given $(\abe),$ $P_{\abe}$ can only differ up to cyclic permutation of the relators or a re-ordering of the generators and relators in the presentation, and we consider all such presentations to be equivalent.

\section{A presentation determines a class of permutation data sets}\label{Ptoabe}

The map sending a group presentation to a permutation data set is not well-defined, and with regard to the properties that we are interested in, all possible $(\abe)$'s determining $P$ are not even ``equivalent,'' since some will allow us to construct a diagram for a closed 3-manifold and some will not.  The fixed, reduced presentation $P$ does determine a finite number of permutation data sets, each of which could determine a different curve set and diagram and therefore a different 3-manifold.  

We consider all $(\abe)$ of degree $d$ that determine a fixed $P$ to be in a class of permutation data sets, which we denote ${\mathbb{P}}_d(P),$ read as ``the class of permutation data sets of degree $d$ for presentation $P$.'' Not all $(\abe)\in{\mathbb{P}}_d(P)$ result in equivalent 3-manifolds, therefore when using signed permutation data to answer whether $P$ has a specific property, we only need to know that \textit{a single} $(\abe)$ in the class ${\mathbb{P}}_d(P)$ has the desired properties.


Let $P=\langle x_1, x_2, \ldots, x_m : r_1, r_2,\ldots, r_n \rangle$ be given.   Let $k_i$ denote the number of times that $x_i$ (or $x_i^{-1}$) appears in the set of relators, and put $\sum k_i=d.$ Let
\begin{align*}
\alpha_1 &= (1, 2, \ldots, k_1),\\
\alpha_2 &= (k_1+1, k_1+2, \ldots, k_1+k_2),\\
& \vdots \\
\alpha_m&=(K+1, K+2,\ldots, d),
\end{align*}
where again $K:=k_1+k_2+\ldots+k_{m-1}$.

A relator $r_j$ is recorded in the cycle $\beta_j$ by assigning each occurrence of a generator $x_i$ a unique entry from $\alpha_i.$ One convention would be to record the lowest numbered element of $\alpha_i$ that has not yet been used for each $x_i$ we encounter, but other conventions will provide the other $(\abe)$'s which determine $P.$ Finally, as we assign a number $i$ for each generator $x_i$ in $r_j,$ record the sign of $x_i$ in $\epsilon(i).$

\begin{exam} 
We now demonstrate how to get one $(\abe)$ that determines $P.$ Consider the presentation for the Heisenberg group $$H_3:=\langle  x,y,z:\, [x,y]z^{-1},\, [x,z],\, [y,z] \rangle .$$
Let $\alpha_1$ correspond to $x,$ $\alpha_2$ correspond to $y,$ and $\alpha_3$ correspond to $z.$ The length of each cycle in $\alpha$ represents the number of times the corresponding generator is used in a relator so we have $$\alpha=\alpha_1\alpha_2\alpha_3=(1,2,3,4)(5,6,7,8)(9,10,11,12,13).$$ 
Record the relators as cycles of $\beta,$ always using the lowest unused number from the appropriate $\alpha_i,$ 
$$\beta =\beta_1\beta_2\beta_3=(1,5,2,6,9)(3,10,4,11)(5,12,6,13).$$
Finally, record the sign of each generator $x_i$ into $\epsilon(i),$
$$\epsilon =(1,-1,1,-1,1,-1,1,-1,-1,1,-1,1,-1).$$
\end{exam}

We force the process of going from $P$ to $(\abe)$ to be well-defined by letting the convention be to record the elements of $\beta_j$ by using up the elements from $\alpha_i$ \textit{consecutively}.  The ill definedness comes from the ambiguity in choosing \textit{some} unused element of $\alpha_i$ for each $x_i$ in $r_j$ while recording $\beta_j.$ The class of permutation data sets ${\mathbb{P}}_d(P)$ is the collection of all such possibilities for $\beta_j.$ 

We next demonstrate how to go from a set of permutation data to a splitting. When we consider \S\ref{Ptoabe} and \S\ref{ribbon} together, we have a way to begin with an arbitrary group presentation, translate $P$ into a pair of signed permutations, and determine properties about the 3-manifold  $M_{\abe}$. 

\section{How a permutation data set determines a diagram}\label{ribbon}

This section is exciting because we begin with the minimum data required to build a 3-manifold. Given two permutations $\alpha, \beta\in S_d,$ and a function $\epsilon:\{1,2,\ldots,d\}\rightarrow \{1,-1\},$ we describe how to build $\sxy$ by constructing the splitting surface $S$ and letting $\alpha$ and $\beta$ define curve sets $X$ and $Y.$ Thus the permutation data set $(\abe)$ is a combinatorial representation of a diagram $D,$ which uniquely determines $M(D).$ We first construct the splitting surface by viewing each permutation as a set of simple, closed, oriented curves crossed with intervals, intersecting the two permutations appropriately, and then capping off the boundary components of this frame. 

So long as $\alpha$ and $\beta$ generate a transitive subgroup in $S_d,$ the ribbon diagram (defined next) will be connected and hence the manifold will be connected.  If $\alpha$ and $\beta$ do not generate a transitive subgroup, then partition the cycles of $\alpha$ and $\beta$ such that each partition is a subset of cycles that do create a transitive subgroup of $S_d,$ renumbering as appropriate and creating new intersection functions that retain the crossing information.  Each new set of permutation data will result in a connected manifold, and the original permutation data set corresponds to a connected sum of these manifolds.  For the rest of this paper, we will assume the permutations generate a transitive subgroup of $S_d.$ 

\subsection{Building the splitting surface from $(\abe)$}\label{abetoS}

Let $(\abe)$ be a permutation data set with $c(\alpha)=m,$ $c(\beta)=n,$ and $\alpha,\beta\in S_d.$ Given $(\abe),$ we construct a \textit{ribbon diagram}, $\tilde{R}_{\abe}$ (or just $\tilde{R}$), as follows:
\begin{enumerate}


\item For each cycle $\alpha_i=(i_1,i_2,\ldots, i_k)$ of $\alpha,$ take an oriented simple closed curve $X_i$ and label points $i_1,i_2,\ldots,i_k$ in order around $X_i.$ Do the same for each cycle $\beta_j$ of $\beta$ to get $Y_j.$


\item Define a \textit{ribbon,} $\tilde{r}(X_i):= X_i \times [0,1],$ or $\tilde{r}(Y_j):= Y_j \times [0,1].$ Notice there is one ribbon for each cycle of a permutation.

\item As each ribbon is now a surface, at each point $i\in\{1,\ldots, d\},$ the oriented interval $i\times [0,1]$ is a normal vector. The ordered pair $(X_i,[0,1])$ puts an orientation on the ribbon surface.


\item For each $i$ choose an oriented interval neighborhood $U_i$ of $i,$ $U_i$ in $X$ and an oriented interval neighborhood $V_i$ of $i,$ $V_i$ in $Y.$ Identify $(U_i\times[0,1])$ with $(V_i\times[0,1])$ by an orientation preserving product homeomorphism 
\begin{itemize}
\item which takes $U_i$ to $(i\times [0,1]),$ and $(i\times [0,1])$ to $+V_i$ if $\epsilon(i)=+1,$ and 
\item which takes $U_i$ to $(i\times [0,1]),$ and $(i\times [0,1])$ to $-V_i$ if $\epsilon(i)=-1.$
\end{itemize}
\item Let $\tilde{R},$ the ribbon diagram, be this identified space.
\end{enumerate}

\begin{figure}[ht!]   
   \begin{center}
            \includegraphics[scale=1.2]{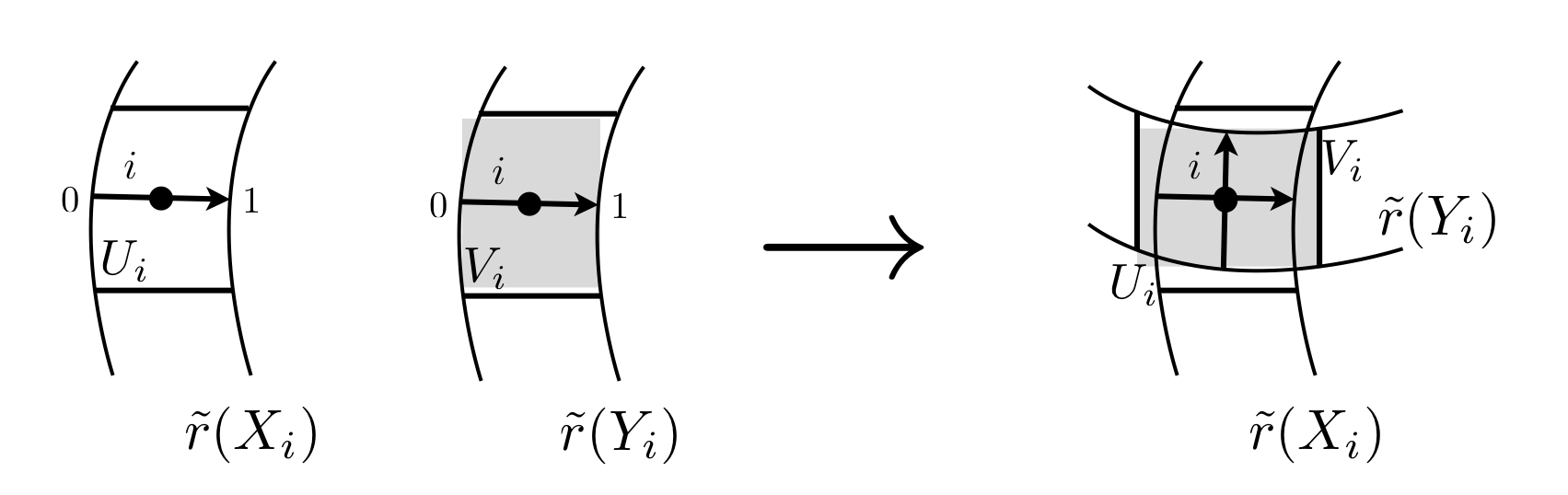}
        \caption{Identifying the neighborhoods $U_i\times [0,1]$ and $V_i\times[0,1],$ hence $\epsilon(i)=+1$}
    \end{center}
\end{figure}

\subsection{Calculating the genus of the splitting surface}

The boundary components of $\tilde{R}$ are polygons, with edges from ribbons $\tilde{r}(X_i)$ and $\tilde{r}(Y_j).$ We obtain $S$ from $\tilde{R}$ by capping off the $b$ boundary components of $\tilde{R}$ with 2-cells, denoted $e^2,$ so  $$S = \tilde{R} \bigcup_{k=1}^b e^2_k.$$  The genus of $S$ is determined from the Euler characteristic of $\tilde{R} \cup e^2_k.$  First note that $\chi(\tilde{R})=-d,$ as $\tilde{R}$ collapses to a graph with $d$ vertices and $2d$ edges. By capping off a boundary component (i.e. adding a 2-cell), we add one to the Euler characteristic. Therefore, by counting the number of boundary components we can calculate the Euler characteristic of $S=\tilde{R} \cup e^2_k$ and the genus of $S.$ 

For $1\leq i\leq d,$ each neighborhood $U_i$ in $\tilde{R}$ occurs at the intersection of some ribbons $\tilde{r}(X_i)$ and $\tilde{r}(Y_j).$ Such an intersection creates four quadrants and each quadrant will appear exactly once in $S$ as the corner of a boundary region.  Quadrants are numbered 1 through 4 as shown  in Figure~\ref{quads}.

\begin{figure}[ht!]   
   \begin{center}
            \includegraphics[scale=1.5]{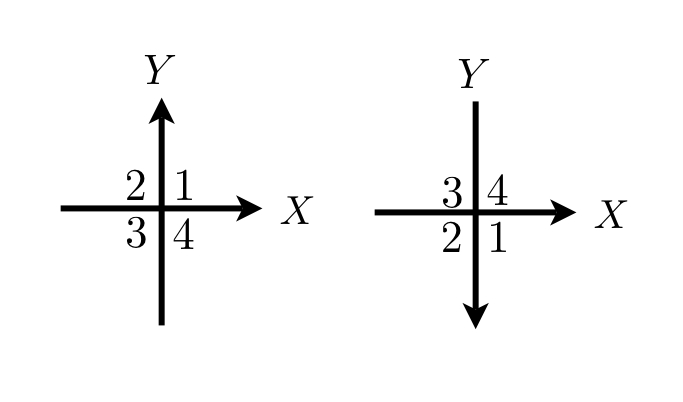}
        \caption{Quadrant labels for a positive crossing (left) and a negative crossing (right)}\label{quads}
    \end{center}
\end{figure}

Beginning at a point $i$ and quadrant $Q,$ we flow along the boundary of $\tilde{R}$ in the counterclockwise direction, recording the $(i,Q)$ corners in that component of $\partial \tilde{R}.$ We get a permutation $\phi$ of the elements $\{(i,Q): i=1,\ldots,d,\,\,\, Q=1,2,3,4\},$ where each boundary component of $\tilde{R}$ corresponds to an orbit of this permutation. 
To count the number of boundary components, we consider the $4d$ ordered pairs $(i,Q),$ and partition them into orbits by noticing the precise rules dictating how one corner flows in the positive direction to the next. The algorithm for partitioning the ordered pairs is next and the 16 possible outcomes are summarized in Table~\ref{table1}.

\noindent \textbf{Algorithm.} \textit{To compute the elements in an orbit}
\begin{enumerate}
\item Beginning in a quadrant, $(i,Q),$ apply the permutation to determine the next corner, $(i',Q').$
\item The map $\phi$ is a separate map on each coordinate:
$\phi(i,Q)=(\phi_1(i, Q,\epsilon(i)), \phi_2(Q,\epsilon(i'))),$
with $\phi_1(i, Q, \epsilon(i))\in \{1,2,\ldots d\},$ and $\phi_2(Q,\epsilon(i'))\in \{1,2,3,4\}.$
\item The intersection number and quadrant determine which edge we flow along to get to the next corner. The edge we flow along from $i$ determines the permutation we apply to $i,$ so as to determine the next corner, $i'.$ The edge $X^{\epsilon}$ corresponds to evaluating $\alpha^{\epsilon}(i),$ and the edge $Y^{\epsilon}$ corresponds to evaluating $\beta^{\epsilon}(i).$ 
\item The new quadrant $Q'$ is determined by the quadrant we started from, $Q,$ and the intersection number of the new corner, $\epsilon(i').$
\item When finished, we will partitioned the $4d$ corners into $b$ orbits, with each orbit corresponding to a boundary component of $\tilde{R}.$
\end{enumerate}

\begin{table}[htdp]
\begin{center}
\begin{tabular}{cccccc}
$\epsilon(i)$ & $Q$ & $\phi_1(i,Q,\epsilon(i))$ & $\epsilon(i')$ & $\phi_2(Q,\epsilon(i'))$ & Illustration\\
\hline
\\
+1 & 1 & $\alpha(i)$ & +1 & 2 & \includegraphics[width=.45in]{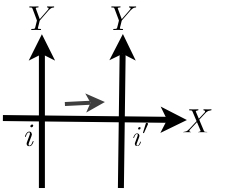} \\
+1 & 1 & $\alpha(i)$ & -1 & 3 & \includegraphics[width=.45in]{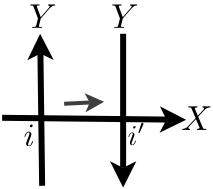}\\
+1 & 2 & $\beta(i)$ & +1 & 3 & \includegraphics[width=.45in]{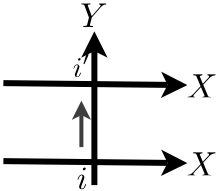}\\
+1 & 2 & $\beta(i)$ & -1 & 4 & \includegraphics[width=.45in]{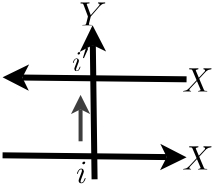}\\
+1 & 3 & $\alpha^{-1}(i)$ & +1 & 4 &  \includegraphics[width=.45in]{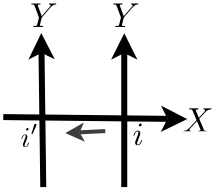} \\
+1 & 3 & $\alpha^{-1}(i)$ & -1 & 1 &  \includegraphics[width=.45in]{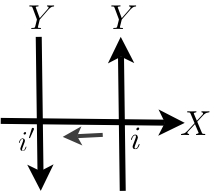} \\
+1 & 4 & $\beta^{-1}(i)$ & +1 & 1 &  \includegraphics[width=.45in]{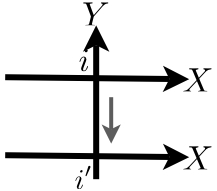} \\
+1 & 4 & $\beta^{-1}(i)$ & -1 & 2 &  \includegraphics[width=.45in]{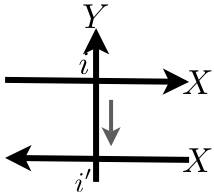} \\
-1 & 1 & $\beta(i)$ & +1 & 3 &  \includegraphics[width=.45in]{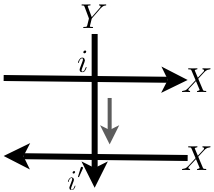} \\
-1 & 1 & $\beta(i)$ & -1 & 4 &  \includegraphics[width=.45in]{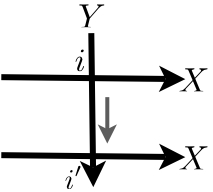} \\
-1 & 2 & $\alpha^{-1}(i)$ & +1 & 4 &  \includegraphics[width=.45in]{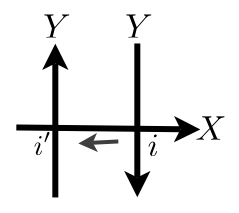} \\
-1 & 2 & $\alpha^{-1}(i)$ & -1 & 1 &  \includegraphics[width=.45in]{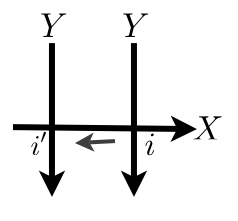} \\
-1 & 3 & $\beta^{-1}(i)$ & +1 & 1 &  \includegraphics[width=.45in]{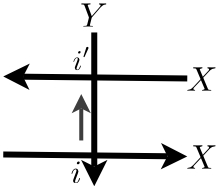} \\
-1 & 3 & $\beta^{-1}(i)$ & -1 & 2 &  \includegraphics[width=.45in]{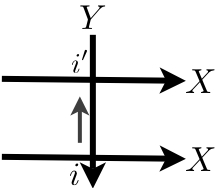} \\
-1 & 4 & $\alpha(i)$ & +1 & 2 &  \includegraphics[width=.45in]{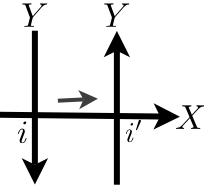} \\
-1 & 4 & $\alpha(i)$ & -1 & 3 &  \includegraphics[width=.45in]{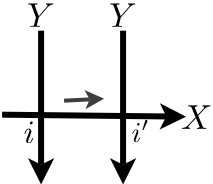} 

\end{tabular}
\caption{Rules for tracing around a component of $S-(X\cup Y)$}\label{table1}
\end{center}
\label{default}
\end{table}

Capping each boundary component of $\tilde{R}$ gives us the closed, oriented splitting surface $S.$ We have a diagram $\sxy,$ with $\chi(S)=\chi(\tilde{R})+b,$ for the $b$ 2-cells attached to $\partial \tilde{R}.$  Thus we have the following lemma.

\begin{lem}
Let $(\abe)$ be given. Then $(\abe)$ determines a diagram $\sxy$ for a 3-manifold, with $S$ the splitting surface of genus
$$g (S) = \frac{1}{2}(b-d+2).$$
\end{lem}

\section{Determining when $M_{\abe}$ is closed and connected}

Beginning with just a pair of signed permutations, we have constructed the diagram $D=\sxy$ for a splitting of $M_{\abe}$ (or just $M$.) We use \textit{planar} to indicate that a surface embeds in ${\mathbb{R}}^2$.  If $S-X$ can be embedded in ${\mathbb{R}}^2$, then it will be a punctured sphere, which can be drawn in the plane (as well as any embedded curves in $S-X$), and hence ``planar.'' The manifold $M$ will be closed if and only if each component of $S-X$ and  of $S-Y$ is planar. In this section, we develop a means of determining  whether  $S-X$  and  $S-Y$ are planar, so as to conclude that  $M_{\abe}$  is closed.

Consider compressing along the disc determined by a simple closed curve, $X_i$ on $S.$ Compression along $X_i$ either separates $S$ or reduces the genus of $S$ by one.  If $S-X$ is one component, then no $X_i$ separates $S$ and each simple closed curve in $X$ that we compress along will reduce $g(S)$ by one. Thus, if $S-X$ is connected and $X$ contains $g(S)$ distinct curves, then $S-X$ is planar. When $S-X$ (or $S-Y$) has more than one component, $M$ will still be closed provided each component is planar. Each non-planar component will determine a boundary component  of $M.$ 

In the last section, given $(\abe)$ we partitioned the $4d$ corners into boundary components of the ribbon diagram, which corresponded to the boundary of the complementary components of $S-(X \cup Y).$ We now place two boundary components of $S-(X \cup Y)$ into an orbit class if they are in the same component of $S-X$. We then calculate the Euler characteristic of each orbit class to determine whether the component is planar. At the end, we will replicate this process to determine whether $S-Y$ has planar components, but for now we are considering only the components of $S-X.$

To create the classes of orbits, each corresponding to a component of $S-X,$ we add the $Y$-curves into $S-(X\cup Y)$ one at a time. As we replace a $Y$-curve, we notice that quadrants on opposite sides of the $Y$-curve will be identified. Hence two components of $S-(X\cup Y)$ will be identified if they were adjacent across a $Y$-curve.  We identify quadrants algebraically by noting that removing a $Y$-curve corresponds to identifying quadrants $1$ and $2,$ and identifying quadrants $3$ and $4,$ regardless of whether a crossing is positive or negative.  That is, place the orbits containing the points $(i,1)$ and $(i,2)$ in the same orbit class, and the orbits containing $(i,3)$ and $(i,4)$ in the same orbit class.
\begin{figure}[ht!]   
   \begin{center}
            \includegraphics[scale=1.7]{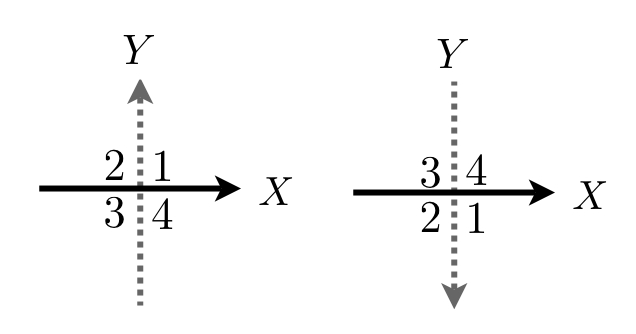}
       \caption{Identifying corners $(i,1)\equiv(i,2)$ and $(i,3)\equiv(i,4)$ when $Y$-curves are replaced}
    \end{center}
\end{figure}

When all $2d$ identifications have been made, the number of orbit classes of boundary components will be the number of components of $S-X.$

When counting the number of components of $S-X,$ it is useful to \textit{tag} each $X$-curve, meaning for each $X_i^{\epsilon},$ choose a corner $(i,Q)$ to track $X_i^{\epsilon}$.

\begin{exam}
  If $\alpha_1=(1,2,3,4)$ and $\epsilon(1)=1,$ then $(1,1)$ can tag $X_1^+$ and $(1,3)$ can tag $X_1^-.$ If $\epsilon(1)=-1,$ then $(1,3)$ could tag $X_1^+$ and $(1,1)$ could tag $X_1^-.$ 
 \end{exam}

Within each component, we will need to count the number of tags, $t,$ because each tag corresponds to one side of an $X$-curve, and we cap each side of each $X$-curve with a 2-cell which affects the Euler characteristic of that component.

The complimentary components of $S-(X\cup Y)$ are polygons, which we denote $p_i.$ Let $[\kappa_x]=\{p_1,p_2,\ldots p_n\}$ denote an orbit class of polygonal components. Then $[\kappa_x]$ corresponds to one component of $S-X,$ call it $\kappa_x,$ once we make the appropriate identifications between the $p_i\in[\kappa_x].$  To calculate the Euler characteristic of $S-X,$ we must calculate the Euler characteristic of each $\kappa_x.$ 

The component $\kappa_x$ is made by identifying $Y$-edges between polygons in $[\kappa_x],$ calculating the Euler characteristic after each identification and finally capping boundary components with discs.  This is equivalent to 
\begin{itemize}
\item totaling the vertices $V$, edges $E,$ and faces $F$ for all of the disjoint polygons $p_i;$ 
\item subtracting two from $V$ and one from $E$ for each $Y$ identification made;
\item adding one to $F$ for each tag in $\kappa_x.$
\end{itemize}
Thus we have the following lemma and corollary.

\begin{lem}
Let $[\kappa_x]=\{p_1, p_2,\ldots, p_n\}$ be an orbit class of polygonal boundary components of $S-(X\cup Y).$  Let $m$ be the number of matched $Y$-edges within the orbit class, and $t$ be the number of tags in $\kappa_x.$ Then the Euler characteristic of this component, is 
$$\chi(\kappa_x)= n - m + t.$$
\end{lem}

\begin{cor}
The component $\kappa_x$ has genus $$g(\kappa_x)=1-\frac{1}{2}(n-m+t).$$
\end{cor}

To determine the total genus of $\partial_X M,$ we merely sum the genus of each component of $S-X.$ 

\begin{lem}\label{BdX}
 Let $\sxy$ be a diagram for $M.$ The genus of $\partial_X M$ is $$g(\partial_X M)=\beta_0(S-X) -\beta_0(X) + g(S)-1.$$
\end{lem}

\begin{proof}
We sum over each $\kappa_x.$ Let $m$ be the number of edge identifications made as we go from $[\kappa_x]$ to $\kappa_x.$
\begin{align*}
g(\partial_XM)&=\sum_{\kappa_x} g(\kappa_x)\\ 
&=\sum_{\kappa_x}(1-\frac{1}{2}(n-m+t))\\
&= \sum_{\kappa_x} 1 - \frac{1}{2}(\sum_{\kappa_x}n -\sum_{\kappa_x}m+\sum_{\kappa_x}t)\\
&= \beta_0(S-X) - \frac{1}{2}(\sum_{\kappa_x}n -\sum_{\kappa_x}m+ 2\beta_0(X) )\\
&= \beta_0(S-X) - \beta_0(X) - \frac{1}{2}(\sum_{\kappa_x}n -\sum_{\kappa_x}m).
\end{align*}

Note that $\sum_{\kappa_x}n$ is the number of components of $S-(X\cup Y),$ and $\sum_{\kappa_x}m=d$ is the number of $Y$-edge identifications that are made.

When we consider the splitting surface $S$ as being built from gluing the cut open surface $S-(X\cup Y)$ along the $X$ and $Y$-curves, we note that 
$$\chi(S)=d-2d+\beta_0(S-(X\cup Y))=\sum_{\kappa_x}n -\sum_{\kappa_x}m,$$ and therefore can write

\begin{align*}
g(\partial_XM)&= \beta_0(S-X) - \beta_0(X) - \frac{1}{2}(\sum_{\kappa_x}n -\sum_{\kappa_x}m)\\
&=\beta_0(S-X) - \beta_0(X) - \frac{1}{2}\chi(S)\\
&=\beta_0(S-X) - \beta_0(X) + g(S) -1.
\end{align*}
\end{proof}

\begin{proof}
\textit{Alternative proof for Lemma~\ref{BdX}.} Let $M=V_X\cup_S V_Y$ be a splitting where $V_X, \,V_Y$ are compression bodies. Let $$M':=U_X\cup_SU_Y,$$ where $U_X$ and $U_Y$ are the same compression bodies, but with only the 2-handles added   and not the 3-handles. Then $$\chi(\partial_XM')=\chi(S)+ \beta_0(X).$$
Since $\partial_XM'$ is equal to $\partial_XM$ with some $2$-spheres added, they have the same genus.  Let $C$ be a component of $\partial_XM'.$ Then we have

\begin{align*}
g(\partial_XM) &=g(\partial_XM')\\
& = \sum_C\frac{2-\chi(X)}{2}\\
& = \beta_0(C) - \frac{1}{2}\chi(\partial_XM')\\
& = \beta_0(S-X) - \frac{1}{2}(\chi(S) + 2\beta_0(X))\\
& = \beta_0(S-X) - \beta_0(X) -1 +g(S).
\end{align*}
\end{proof}

\begin{cor}\label{emptyX}
The manifold, $M$, determined by the diagram $\sxy$ has $\partial_XM=\emptyset$ when 
$$g(S)=1+ \beta_0(X) - \beta_0(S-X).$$
\end{cor}

If $c(\alpha)=g(S),$ then we know $P_{\abe}$ presents the group $\pi_1(M_{\abe}).$

To count the number of components of $S-Y,$ we duplicate this process, relabeling $X$ and $Y$ where appropriate. The only significant difference when calculating the Euler characteristic of the components of $S-Y,$ is that adding $X$-curves into $S-(X\cup Y),$ corresponds to identifying quadrants $1$ and $4,$ and identifying quadrants $2$ and $3.$
\begin{figure}[ht!]   
   \begin{center}
            \includegraphics[scale=1.7]{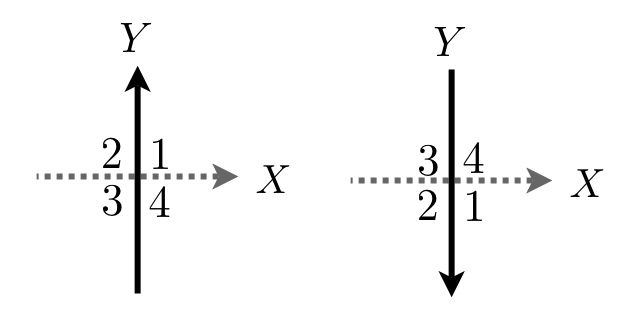}
        \caption{Identifying corners $(i,1)\equiv(i,4)$ and $(i,2)\equiv(i,3)$ when $X$-curves are replaced}
    \end{center}
\end{figure}

Corollary~\ref{emptyX} and the equivalent statement for $\partial_YM$ give us a closed condition for a 3-manifold.

\begin{cor}\label{Closed}
The manifold $M$ determined by the diagram $\sxy$ is closed if and only if 
$$g(S)=1+ \beta_0(X) - \beta_0(S-X)$$
and
$$g(S)=1+ \beta_0(Y) - \beta_0(S-Y).$$
\end{cor}

\vspace{12pt}	
In \S\ref{Ptoabe}--\ref{ribbon}, we determined an algorithm that begins with a presentation, converts the presentation to a permutation data set, and then converts the permutation data set to a diagram, allowing us to determine some nice properties about the 3-manifold determined by $P$. Going from a presentation to a diagram was already a well understood process \cite{Zi}, but the intermediate step of converting to $(\abe)$ gives us a systematic means of understanding the association between a presentation and the 3-manifold.   Section~\ref{ribbon} carefully outlined the process for creating a diagram from a permutation data set, and \S7 showed, as promised, how one could determine if the 3-manifold associated to such a presentation is closed.


\end{document}